\documentclass[12pt]{article}
\usepackage{latexsym}
\usepackage{amsfonts}
\usepackage{amsmath}
\usepackage{amsmath,amssymb,proof1,latexsym}
\usepackage[all]{xypic}

\makeindex

\newcommand{\imp}{\!\rightarrow\!}

\newcommand{\mpn}{\medskip\par\noindent}

\newcommand{\pa}{{\sf PA}}

\newcommand{\proves}{\vdash}

\newcommand{\gn}[1]{\ulcorner {#1} \urcorner }

\newcommand{\lc}[1]{#1\!\!:\!\!}

\newtheorem{Prop}{\bf Proposition}

\newtheorem{Theor}{\bf Theorem}

\newtheorem{Lemma}{\bf Lemma}

\newtheorem{Coro}{\bf Corollary}

\newtheorem{Fact}{\bf Fact.}

\newtheorem{Remark}{\bf Remark}

\newtheorem{Claim}[enumi]{Claim}

\newtheorem{defin}{\bf Definition}
\newenvironment{definition}{\begin{defin} \em}{\end{defin}}
\newtheorem{exam}{\bf Example}

\newtheorem{notat}{\bf Notation.}

\newenvironment{proof}{{\bf Proof.}}{\hfill $\slot$}
\newcommand{\slot}{\hfill \mbox{$\Box$}\vspace{\parskip}\\}
\newtheorem{Comment}{\bf Comment}

\begin{document}

\title{Serial Properties, Selector Proofs, and \\ the Provability of Consistency.}

\author{Sergei Artemov\\ \\
 {\small The Graduate Center, the City University of New York}\\
{\small  365 Fifth Avenue, New York City, NY 10016}\\
{\small {\tt sartemov@gc.cuny.edu}} }
\date{\today}
\maketitle

\begin{abstract}
For Hilbert, the consistency of a formal theory $T$ is an infinite series of statements ``$D$ is free of contradictions" for each derivation $D$ and a consistency proof is i) an operation that, given $D$, yields a proof that $D$ is free of contradictions, and ii) a proof that (i) works for all inputs $D$. Hilbert's two-stage approach to proving consistency naturally generalizes to the notion of a finite proof of a series of sentences in a given theory. Such proofs, which we call selector proofs, have already been tacitly employed in mathematics. Selector proofs of consistency, including Hilbert's epsilon substitution method, do not aim at deriving the G\"odelian consistency formula ${\sf Con}_T$ and are thus not precluded by G\"odel's second incompleteness theorem. We give a selector proof of consistency of Peano Arithmetic \pa\ and formalize this proof in \pa. 
\end{abstract}

\section{Introduction}
\subsection{Hilbertian consistency}\label{Hcon}
In his foundational works, Hilbert viewed a consistency proof as follows:
\begin{quote}
``What is required for a consistency proof is an operation which, given a formal
derivation, transforms such a derivation into one of a special form, plus proofs
that the operation in fact succeeds in every case and that proofs of the special
kind cannot be proofs of an inconsistency."(\cite{Zach07}, cf. also \cite{Zach16})
\end{quote}
This quote characterizes the format of Hilbert's {\it epsilon substitution method}, cf. \cite{HB34}. It consists of an operation which given a derivation $D$ produces evidence, a chain of reductions to some normal form, that $D$ is free of contradictions followed by a proof that this operation terminates with success for each input $D$. 
Two fundamental observations can be made here:
\renewcommand{\labelenumi}{\text{\roman{enumi})}}
\begin{enumerate}
\item Consistency of a formal theory $T$ is treated as the set of claims 
\begin{equation}\label{defcon}
\mbox{``$D$ is not a derivation of a contradiction in $T$"} 
\end{equation}
with a parameter $D$ ranging over derivations in $T$. In terms of G\"odel codes, consistency is the series of arithmetical sentences 
\begin{equation}\label{set}
C(0),\ C(1),\ C(2),\ \ldots\ \  
\end{equation}
with $C(n)$ being 
$
\mbox{\it ``$n$ is not a code of a derivation of a contradiction in $T$."} 
$
\item A proof of consistency is two-stage: an operation (which we suggest calling {\it selector}) that given $D$ certifies that $D$ is contradiction-free plus a verification of this operation.  
\end{enumerate}
Whereas proofs $D$ are necessarily formal, the verifications from (ii) may be contentual pieces of mathematical reasoning. They need not necessarily be formal. 

Note that this view of consistency is quite different from a popular later representation of consistency of $T$ by a single arithmetical formula 
${\sf Con}_T$:
\begin{equation}\label{sformula}
\forall x C(x).
\end{equation}
Indeed, (\ref{set}) and (\ref{sformula}) are equivalent in the standard model of arithmetic. This engenders the temptation to use the formula (\ref{sformula}) in lieu of the series (\ref{set}) for the analysis of the provability of consistency. This would be a mistake, however, since the provability in $T$ requires considering all models of $T$ and there (\ref{set}) and (\ref{sformula}) are not equivalent. In particular, for $T=\pa$, (\ref{set}) holds in any model, but  (\ref{sformula}) fails in some models of \pa. 

We are not claiming that Hilbert himself explicitly formulated (i). However, in his works, he acted according to (i). Hilbert's epsilon substitution method aims at proving
\[ \mbox{\it ``for all $D$ there is a proof that $D$ is consistent"} \]
which corresponds to format (\ref{set}) and is mathematically and conceptually different from establishing 
\[
\mbox{\it ``there is a proof that for all $D$, $D$ is consistent"}
\]
corresponding for format (\ref{sformula}). So, 
 Hilbert's approach to proving consistency fits to consistency as a series (\ref{set}) and does not fit to consistency as a formula (\ref{sformula}). 





\subsection{Proving serial properties}\label{HvsG}
We extend Hilbert's approach to proving consistency to a general definition of a proof of a serial property in a given theory.
By a {\it serial property} ${\cal F}$ we understand a series of claims $\{ F_0, F_1,\ldots,F_n,\ldots \}$.  
Suppose, each of $F_n$ is a sentence in the language of some theory $\cal S$. 
What is the right mathematical notion of a proof of $\cal F$ in $\cal S$? 
\medskip\par
One might think this is the instance provability:  
\begin{equation}\label{instance}
\mbox{\it for each n, $\cal S$ proves $F_n$}.
\end{equation}
It, however, is not.  The instance provability alone
is too weak since it does not address the issue of whether (\ref{instance}) itself is provable in $\cal S$. 
For example, the consistency proof for  \pa\ via truth in the standard model yields instance provability of the consistency (\ref{set}):
\begin{quote} {\it Let $D$ be a formal derivation in \pa. Since all formulas from $D$ are true in the standard model and $(0\!=\!1)$ is not true, the latter is not in $D$. 
Therefore $C(n)$ holds for each numeral $n$, hence it is provable in \pa\ as a true primitive recursive statement.}
\end{quote}
However, the above is not a proof in \pa\ since the notion ``true in the standard model" is not formalizable in \pa. 
So, in addition to instance provability, some sort of verification of (\ref{instance}) by means of $S$ is also needed.
\medskip\par
Quantification on parameters is too strong for proving serial properties. 
Let  ${\cal F}$ be a serial property $\{F(0),F(1),\ldots, F(n),\ldots\}$ for some arithmetical formula $F(x)$. What is wrong with assuming that  ``$\cal F$ is provable in \pa" means
\begin{equation}\label{quant}
\pa\ \proves\ \forall x F(x)? 
\end{equation}
In \pa, formula $\forall x F(x)$ is strictly stronger than $\cal F$, since $\forall x F(x)$ yields all of $\cal F$, but $\pa + {\cal F}$ does not necessarily yield $\forall x F(x)$, cf. Proposition \ref{p3}, Section~\ref{shemy}.  As a result, using $(\ref{quant})$ for establishing provability of the property $\cal F$ in \pa\ is a correct and convenient simplification: once $ \pa\ \proves\ \forall x F(x)$, \pa\ verifiably proves all $F(n)$'s.   
The picture is quite different when we try using the failure of $(\ref{quant})$ to claim the unprovability of a serial property. 
The negation of $(\ref{quant})$ is too weak to yield the unprovability of $\cal F$,
\[ \mbox{\it $\pa\not\proves\forall x F(x)$ does not imply that {\cal F} cannot be established by means of \pa.} \]
Indeed, one could {\em a priori} imagine proving each $F(n)$ in a way uniformly verifiable in \pa. 
This would establish the property $\cal F$ in \pa\ without proving $\forall x F(x)$. 
For these reasons, we suggest that using $\forall x F(x)$ in lieu of $\cal F$ is a strengthening fallacy. 
The approach $(\ref{quant})$ cannot be used for establishing unprovability of $\cal F$ in \pa.
\medskip\par
What about using implicit selectors represented by quantifiers rather than constructive operations? Within this approach, for an arithmetical formula $F(x)$ and a serial property ${\cal F}=\{F(0),F(1),\ldots, F(n),\ldots\}$, ``{$\cal F$ is provable in \pa}" means
{``\pa\ proves that each instance of $\cal F$ is provable in \pa"}
which is 
naturally formalized as 
\begin{equation}\label{foralle}
\pa\proves\forall x\Box F^\bullet(x).
\end{equation}
Here ``$\Box$" is the provability predicate in \pa\ (with the suppressed existential quantifier over proofs), cf. Section~\ref{s2},  and $F^\bullet(x)$ is a natural primitive recursive (p.r.) term which for each $n$ returns the G\"odel number $\gn{F(n)}$ of $F(n)$. This approach requires consistency assumptions even for basic tasks. For example, in addition to (\ref{foralle}), it takes 1-consistency of \pa\ to prove the instance provability of $\cal F$, i.e., that for each $n=0,1,2\ldots$, 
\[ \mbox{\it $\pa\proves F(n)$}.
\]
For this reason, approach (\ref{foralle}) does not fit for analyzing provability of consistency. Selector proofs (Section~\ref{sel}) avoid this deficiency by using explicit selector function ``$s(x)$" instead of the existential quantifier on proofs in ``$\Box$":
\begin{equation}\label{forallex}
\pa\proves\forall x[\lc{s(x)}F^\bullet(x)].
\end{equation}
Here ``$u\!:\!w$" is a short for the arithmetical formula {\it ``$u$ is proof of $w$ in $\pa$}," cf. Section~\ref{s2}. Now, (\ref{forallex}) yields that $\pa\proves F(n)$ for each $n=0,1,2,\ldots$, cf.~Proposition~\ref{p2}(ii), Section~\ref{shemy}.

\subsection{Selector proofs}\label{sel}

We assume that formal theories $T$ which we consider have contentual versions $\widehat{T}$ with the same mathematical axioms as $T$. Each formula in the language of $T$ may be regarded as a statement in the language of $\widehat{T}$ as well.  Let $T$ be a formal theory and $D$ a formal derivation of $\varphi$ in $T$. Then we regard $D$ to be a proof of $\varphi$ in $\widehat{T}$ as well. The converse also holds: a given rigorous proof in $\widehat{T}$ may be formalized as a derivation in $T$. 

The following definition represents Hilbert's approach to proving consistency generalized to all serial properties. 
Let  ${\cal F}$ be a serial property $\{ F_0, F_1,\ldots,F_n,\ldots \}$. 

{\it A proof of ${\cal F}$ in a theory $\cal S$} is a pair of
\renewcommand{\labelenumi}{\text{\roman{enumi})}}
\begin{enumerate}
\item a {\it selector} which is an operation that given $n$ provides a proof of $F_n$ in $\cal S$;
\item a {\it verifier} which is a proof in $\cal S$ that the selector does (i). 
\end{enumerate}
We call such pairs (i) and (ii) {\it selector proofs}. 
For the purposes of this work, selectors are assumed to be explicit primitive recursive operations on finite objects. 

This definition works for both formal theories and their contentual versions. If $\cal S$ is a contentual counterpart $\widehat{T}$ of some formal theory $T$, the selector builds a formal derivation of $F_i$ in $T$ to represent a contentual proof of $F_i$ in ${\cal S}$ (which is $\widehat{T}$).
When safe, we allow syntactic parameters other than numerals, including sentences $F_i$ themselves as selector inputs. Furthermore, we routinely identify syntactic objects, such as terms, formulas, derivations, etc., with their numeral codes. 

In Section~\ref{examples}, we provide a body of examples that demonstrate that selector proofs have been  tacitly employed by mathematicians and play a prominent role in metamathematics.

 {\it A natural formalization in \pa\ of a given selector proof of ${\cal F}$ in ${\cal S}$} is a pair $\langle s,v\rangle$ with: 
\renewcommand{\labelenumi}{\text{\roman{enumi})}}
\begin{enumerate}
\item an arithmetical term $s(x)$ formalizing the given selector procedure, 
\item a natural formalization of a given verifier which is a \pa-proof $v$ of 
\[ \forall x [s(x)\!:_{\cal S}\! F^\bullet(x) ].\] 
\end{enumerate}
Here ``$u\!:_{\cal S}\!w$" is a short for the arithmetical formula {\it ``$u$ is proof of $w$ in $\cal S$"} and $F^\bullet(x)$ is a natural p.r. term such that $F^\bullet(n) = \gn{F_n}$. This definition naturally extends from \pa\ to any other sufficiently strong coding theory.

Consider the Brouwer-Heyting-Kolmogorov  treatment of universal arithmetical statements adjusted by so-called Kreisel's ``second clause"  (\cite{DK16,Kre62,Tro69,vD73}, cf. also \cite{AF19}, Section 10.2.1).
{\it A constructive proof} of $\forall x F(x)$ is a pair $\langle s,v\rangle$ where 
\begin{itemize}
\item $s(x)$ is an operation on proofs and 
\item $v$ is a proof that for each $x$, $s(x)$ is a proof of $F(x)$. 
\end{itemize}
So, conceptually, Brouwer-Heyting-Kolmogorov proofs of universal statements are selector proofs.

\subsection{Selector proofs in mathematics}\label{examples}
 
{\bf Example 1.}  Complete Induction principle, $\cal CI$: \textit{for any formula $\psi$,}
\[ \mbox{\it if for all $x$ $[\forall y<x\ \psi(y)$ implies $\psi(x)]$, then $\ \forall x\psi(x)$}.\]
Here is a textbook proof of $\cal CI$: {\it apply the usual \pa-induction to $\forall y< x\ \psi(y)$ to get the $\cal CI$ statement ${\cal CI}(\psi)$ for $\psi$.} 
\mpn
This is a contentual selector proof which, given $\psi$ selects a derivation of ${\cal CI}(\psi)$ in a way that obviously works for any input $\psi$. If needed, this proof can be made formal:
\begin{itemize}
\item pick a primitive recursive selector term $s(x)$ which given $\gn{\psi}$ computes the code of a derivation in \pa\ of ${\cal CI}(\psi)$ above;
\item find an easy proof in \pa\ (verifier) that $s(x)$ works for all inputs\footnote{For convenience, in this and similar cases, assume that the coding is surjective. Here this means that each numeral $n$ is a code of some $\psi$.}
\[ \pa\proves\forall x[s(x)\!:\!C\!I^{\bullet}(x)]\] 
with $C\!I^{\bullet}(x)$ a natural p.r. term such that $C\!I^{\bullet}(\gn{\psi})=\gn{{\cal CI}(\psi)}$. 
\end{itemize}
\mpn
{\bf Example 2}.  
{\em The product of polynomials is a polynomial.}\footnote{To be precise, by polynomial in arithmetic we understand a term 
$a_nx^n+a_{n-1}x^{n-1}+\ldots+a_1x+a_0$ 
with some default rule of placing the parentheses, where each $a_i$ is a numeral and $x$ a variable. In the product $f\cdot g$ of polynomials $f$ and $g$, ``$\cdot$" stands for the usual multiplication in \pa.}
\mpn
Here is its standard mathematical proof:
{\it given a pair of polynomials $f,g$, using the well-known formula, calculate coefficients of the product polynomial $p_{f\cdot g}$, and prove in arithmetic that}
\begin{equation}\label{f.g}
f\!\cdot\! g = p_{f\cdot g}.
\end{equation}
This is a contentual selector proof: for each $f,g$, it finds a proof of (\ref{f.g}) in \pa.\footnote{Proving (\ref{f.g}) in \pa\ by defining the product of polynomials $f$ and $g$ as ``the normal form of the \pa-product $f\!\cdot\! g$" simply hides the same selector procedure.}
To formalize this proof in \pa, we can build 
\begin{itemize} 
\item a p.r. term $\mbox{\it Product}^\bullet (x,y)$ such that $\mbox{\it Product}^\bullet (\gn{f},\gn{g})$ is $\gn{f\! \cdot\! g\!=\! p_{f\cdot g}}$;  
\item a natural p.r. selector $s(x,y)$ such that 
$s(\gn{f},\gn{g})$ is the G\"odel number of a \pa-derivation of   
$f\! \cdot\! g\!=\! p_{f\cdot g}$. 
\end{itemize}
By direct formalization of the above reasoning, \pa\ proves
\[ \forall x,y [s(x,y)\!:\!\mbox{\it Product}^\bullet(x,y)].\] 
{\bf Example 3}. Take the double negation law, DNL\footnote{or any other tautology containing propositional variables.}: \textit{for any arithmetical formula $X$, }
\vskip-5pt
\begin{equation}\label{dML}
X\imp \neg\neg X. 
\end{equation}
The standard proof of DNL in \pa\ is {\it for given $X$, build the usual logical derivation $D(X)$ of {\em (\ref{dML})} in \pa. }
This is a contentual selector proof which builds an individual \pa-derivation for each instance of DNL in a way that provably works for any input $X$. 
This proof can be easily formalized in \pa\ as a derivation of 
\[
\forall x [s(x)\!:\!{\it DNL}^{\bullet}(x)].
\]
Here ${\it DNL}^\bullet(x)$ is a natural coding term such that 
${\it DNL}^\bullet(\gn{X})=\gn{X\imp \neg\neg X}$.
The selector $s(x)$ is a natural p.r. term such that $s(\gn{X})=\gn{D(X)}$. 
 
\medskip\par
So, selector proofs have been used as proofs in arithmetic of serial properties 
\[ \{ F(u)\} \]
in which $u$ is a syntactic parameter (ranging over terms, formulas, derivations, etc.) and $F(u)$ is an arithmetical formula for each $u$. Given $u$, such a  proof selects a \pa-derivation for $F(u)$ with subsequent, possibly informal or default, verification of this procedure. 
Within this tradition, we have to accept selector proofs as legitimate logic objects. In this case, the proof of \pa's consistency within \pa\ from Section~\ref{conproof} is a natural consequence.

If we don't accept selector proofs, we have to admit that such basic mathematical facts as induction, propositional tautologies, product of polynomials, etc., are not provable in arithmetic. Moreover, that consistency is unprovable for the same reason: selector proofs are not included.

\section{G\"odel's incompleteness and Hilbert's program}\label{s2}

Some notational conventions. 
\mpn
$\bullet$ We use ``$\bot$" to denote a standard arithmetical contradiction $(0\!=\!1)$. 
\mpn
$\bullet$ $\mbox{${\it Proof}(x,y)$ (or ``$x\!:\!y$" for short)}$
is the primitive recursive {\it proof predicate} in \pa:
\[ \mbox{{\it ``$p$ is a \pa-proof of $\varphi$"} means $\gn{p}\!:\!\gn{\varphi}$}. \]
We will omit $\gn{\ }$'s when safe and write $\gn{p}\!:\!\gn{\varphi}$ as $p\!:\!\varphi$. 
\mpn
$\bullet$ {\it The provability predicate} ${\it Provable}(y)$ is $\exists x(\lc{x}y)$. We also write $\Box\varphi$ for ${\it Provable}(\varphi)$.
\mpn
$\bullet$ The {\it \pa-consistency formula} ${\sf Con}_{\sf PA}$ can be written in either notation:
\[ \mbox{$\neg{\it Provable}(\bot)\ \ $ or $\ \ \forall x\neg{\it Proof}(x,\bot)\ \ $ or $\ \ \forall x\neg \lc{x}\bot\ \ $ or $\ \ \neg\Box\bot$. } \]

G\"odel's second incompleteness theorem, {\bf G2} (\cite{God31}) yields that 
${\sf Con}_{\sf PA}$
is not derivable in \pa\ (given that \pa\ is consistent). 
Together with the widely accepted {\it Formalization Principle}, {\bf FP}, stating 
``any rigorous reasoning within the postulates of \pa\
can be formalized as a derivation in \pa,"
this prompted a popular opinion that 
\begin{equation}\label{main}
 \mbox{{\bf G2} + {\bf FP} yields that 
\pa\ cannot prove its own consistency.} 
\end{equation}
We argue that (\ref{main}) is not warranted. Indeed, suppose there is a proof of the \pa-consistency in $\widehat{\pa}$. By {\bf FP}, it can be formalized as a derivation in \pa. However, we cannot claim that this formalization is necessarily a \pa-derivation of the given formula ${\sf Con}_{\sf PA}$.  The problem is that there might be nonequivalent formalizations of \pa-consistency and its proofs. 

Hilbert's epsilon substitution method, HES, is a selector proof toolkit and selector proofs do not derive standard consistency formulas. Imagine that HES were successful in proving the consistency of \pa\ in \pa. This would yield a \pa-proof of the arithmetical formula 
${\sf HCon}_{\sf PA}$ stating  
\begin{quote}
{\it ``for each derivation $D$ the epsilon-term elimination process on $D$ terminates with success and produces a \pa-proof that $D$ does not contain $\bot$"}
 \end{quote}
rather than a \pa-proof of ${\sf Con}_{\sf PA}$. 
Moreover, there are no reasons to assume that ${\sf HCon}_{\sf PA}$ would yield ${\sf Con}_{\sf PA}$ in \pa. 
In a similar situation, the successful selector proof of \pa-consistency in Section~\ref{conproof} produces a \pa-derivation of the formula ${\sf SCon}_{\sf PA}$ stating
\begin{quote}
{\it ``for each derivation $D$ the given selector produces a \pa-proof that $D$ does not contain $\bot$."}
 \end{quote}
Such ${\sf SCon}_{\sf PA}$ is provable in \pa, but ${\sf Con}_{\sf PA}$ is not. 
\medskip\par
The above observations allow us to conclude that
\begin{enumerate}
\item Despite the predominant opinion, HES is not precluded by G\"odel's second incompleteness theorem. This is because HES does not attempt to prove the consistency formula ${\sf Con}_{\sf PA}$. Hilbert himself, in his introduction to \cite{HB34} in 1934, stated that G\"odel's work did not refute his consistency program. Here is the English translation of this quote taken from Feferman's \cite{Fef08}: 
\begin{quote}
``... I would like to emphasize the following: the view ... which maintained that certain recent results of G\"odel show that my proof theory can't be carried out, has been shown to be erroneous. In fact that result shows only that one must exploit the finitary standpoint in a sharper way for the farther reaching consistency proofs."
\end{quote}
Only recently (\cite{Kri22}), Kripke argued that HES could not succeed in proving consistency of \pa\ in \pa\ by reasons not related to G\"odel's second incompleteness theorem. 
\item Contrary to popular opinion, the unprovability of consistency paradigm in the form 
\begin{quote}
``there exists no consistency proof of a system that can be formalized in the system itself" (\cite{EB20}, {\it Encyclop{\ae}dia Britannica}) 
\end{quote}
has been unwarranted. Moreover, this paradigm is now proven to be incorrect: selector proofs which can be traced back to Hilbert, have eventually succeeded in proving the consistency of \pa\ in \pa, cf. Section~\ref{conproof}.
\end{enumerate}
This gap in the unprovability of consistency argument can be illustrated by the following question asked by M.~Vardi  in 2021 in a private discussion on the provability of consistency paper \cite{Art19}.
\mpn
{\bf Example 4}. Let ${\sf Con}_{T}$ be the consistency formula for an arithmetical theory $T$: \\ 
$\forall x\ \mbox{\it ``$x$ is not a $T$-derivation of $\bot$."}$
Consider theories:
\[ \pa_0=\pa,\ \ \ \ \ \  \pa_{i+1}=\pa_i+{\sf Con}_{{\sf PA}_i},\ \ \ \ \ \ \pa^\omega=\bigcup_{i\in\omega}\pa_{i}
 .\]
{\bf Q}: What is wrong with the following consistency proof of $\pa^\omega$ by means of $\pa^\omega$?
\begin{quote}
{\it Let $D$ be a derivation in $\pa^\omega$. Find $n$ such that  $D$ is a derivation in $\pa_{n}$. ${\sf Con}_{{\sf PA}_{n}}$ -- one of the postulates of $\pa^\omega$ -- implies that $D$ does not contain $\bot$.}
\end{quote}
{\bf A}: This is a fine selector proof of consistency of $\pa^\omega$ in $\pa^\omega$. The selector that, given $D$, computes the code of a $\pa^\omega$-derivation of {\it ``$D$ does not contain $\bot$"}  is primitive recursive and so simple that its verification in \pa\ is straightforward. 

On the other hand, by {\bf G2},  
$\pa^\omega$ cannot prove the consistency formula ${\sf Con}_{\pa^\omega}$.  This observation, however,  does not harm the above consistency proof because that proof does not prove ${\sf Con}_{\pa^\omega}$. A complete formalization in \pa\ of the consistency proof from this example is a derivation of the formula 
\[ \forall x [s(x)\!:_{\omega}\!\neg x\!:_{\omega}\bot]
\]
with ``$u\!:_{\omega}\!v$" being a proof predicate in $\pa^\omega$ and $s(x)$ a natural term for the aforementioned primitive recursive selector function.

In Section~\ref{conproof}, we offer a selector proof of \pa-consistency in \pa\ (first appeared in \cite{Art19}).
This undermines the aforementioned ``unprovability of consistency paradigm." 

\section{Proving \pa-consistency in \pa}\label{conproof}



We provide a selector proof of \pa-consistency in $\widehat{\pa}$ (the contentual version of \pa) and 
then formalize this selector proof in \pa. The principal idea of the proof is to internalize the \pa-reflexivity proof in \pa. Here is a proof sketch.
\medskip\par
{\bf Step 1.} Inspect the proof of reflexivity of \pa\ (cf. \cite{HP17}):

\begin{quote}
{\it Let $\pa\!\upharpoonright_n$ be the fragment of \pa\ with the first $n$ axioms. Then for each numeral $n$, $\pa\vdash \mbox{\sf Con}_{\pa\upharpoonright_n}.$}
\end{quote}
From this proof, extract a natural primitive recursive function that given $n$ builds a  derivation of $\mbox{\sf Con}_{\pa\upharpoonright_n}$ in \pa. Note that $\mbox{\sf Con}_{\pa\upharpoonright_n}$ yields in \pa\ that
\[\mbox{\it ``$D$ is not a proof of $\bot$" }\]
for any specific derivation $D$ in $\pa\!\upharpoonright_n$.

{\bf Step 2} -- a contentual selector proof of the consistency of \pa. Given a \pa-derivation $D$, find $n$ such that $D$ is a derivation in $\pa\!\upharpoonright_n$ (an easy primitive recursive procedure). By Step 1, \pa\ proves {\it ``$D$ is not a proof of $\bot$"} and the code of this proof can be calculated by a primitive recursive selector $s(x)$ from the code of $D$. 

{\bf Step 3} -- internalization. A natural internalization of Step 1 and Step 2 in \pa\ yields the desired verification of the selector in \pa:
\[ \pa\proves \forall x[s(x)\!:\! \neg x\!:\!\bot ].\]
\subsection{A more detailed proof}
A more detailed proof of consistency for \pa\ in \pa\ goes by two consecutive stages.  
\begin{enumerate}
\item A contentual selector proof of \pa-consistency in its original form (\ref{defcon}). 
Given a \pa-derivaton $D$, we find an arithmetical invariant ${\cal I}_{D}$,
establish that ${\cal I}_{D}(\varphi)$ holds for each $\varphi$ in $D$,
${\cal I}_{D}(\bot)$ does not hold, hence $\bot$
does not occur in ${D}$. Note that {\bf G2} prohibits having such an invariant ${\cal I}$ uniformly for all derivations $D$ but does not rule out having an individual invariant ${\cal I}_D$ for each $D$.
\item A comprehensive formalization of (i) in \pa: define the invariants ${\cal I}_{D}$ in \pa, formalize the selector $s(x)$ as a p.r. term and prove in \pa\ that for each $x$, $s(x)$ is the code of a  \pa-proof of  {\it ``$x$ does not contain $\bot$."}
\end{enumerate}
We will be reasoning in $\widehat{\pa}$.

\medskip\par
In the metamathematics of first-order arithmetic, there is a well-known construction called \textit{partial truth definitions} 
(cf. \cite{Buss98, HP17, Pud98, Smo85}). Namely, for each $n=0,1,2,\ldots$ we build, in a primitive recursive way, a $\Sigma_{n+1}$-formula 
\[ \mbox{\it Tr}_n(x,y), \]
called \textit{truth definition for $\Sigma_{n}$-formulas}, which satisfies natural properties of a truth predicate.
Intuitively, if $\varphi$ is a $\Sigma_{n}$-formula and $y$ is a sequence encoding values of the parameters in $\varphi$, then 
$\mbox{\it Tr}_n(\varphi,y)$ defines the truth value of $\varphi$ on $y$. We allow ourselves to drop ``$y$" and to write $\mbox{\it Tr}_n(\varphi)$ when $\varphi$ does not have parameters. The following proposition is well-known in proof theory of \pa.
\begin{Prop}\label{p1} {\it For any $\Sigma_{n}$-formula $\varphi$, the following Tarksi's condition holds and is provable in \pa\ (and in $\widehat{\pa}$):
\begin{equation}\label{Tarski}
\mbox{\it Tr}_n(\varphi,y)\ \leftrightarrow\ \varphi(y).
\end{equation} 
In particular, \pa\ (as well as $\widehat{\pa}$) proves $\neg\mbox{\it Tr}_n(\bot)$.}
\end{Prop}
\mpn
The proof of Proposition~\ref{p1} yields a natural primitive recursive function which given $n$ and $\varphi$ calculates the code of a \pa-proof of $(\ref{Tarski})$.
Note that proofs in Proposition~\ref{p1} are rigorous mathematical arguments without any meta-mathematical assumptions about \pa. 

\medskip\par
Given a \pa-derivation $D$, build a proof in $\widehat{\pa}$ that $D$ does not contain $\bot$.
\begin{enumerate}
\item take the conjunction $\widetilde{D}$ of the universal closures of all formulas from $D$ and calculate large enough $n$ to satisfy conditions of Proposition~\ref{p1} for $\widetilde{D}$ as $\varphi$;
\item by Proposition~\ref{p1}, find a proof of $\widetilde{D} \rightarrow \mbox{\it Tr}_n(\widetilde{D})$;
\item build an easy proof of $\widetilde{D}$;
\item from (ii) and (iii), generate a proof of $\mbox{\it Tr}_n(\widetilde{D})$; 
\item by Proposition 1, find a proof of $\neg\mbox{\it Tr}_n(0\!=\!1)$;
\item from (iv) and (v), conclude that $D$ does not contain $\bot$. 
\end{enumerate}
This is a mathematically rigorous selector proof of consistency of \pa\ in $\widehat{\pa}$. It describes a p.r. selector function $s$ which given $D$ produces a proof in $\widehat{\pa}$ of the statement {\it ``$D$ does not contain $\bot$."} 
This proof is step-by-step formalizable in \pa.  

\hfill $\Box$

\section{Further observations}\label{discussion} 

As was shown, \pa\ selector proves the consistency of \pa. Whether  \pa\ can selector prove the consistency of $\pa+ {\sf Con}_{\sf PA}$?  As was noticed by Kurahashi~\cite{Kur19} and Sinclaire \cite{Sin19}, independently, the answer is ``no."

Let  $\lc{x}_{\mbox{\it\tiny T}}\ \!\varphi$ be the natural primitive recursive proof predicate in a theory $T$ and 
$\Box_{\mbox{\it\tiny T}}\ \!\varphi$ denote {\it``$\varphi$ is provable in $T$,"}  i.e., $\exists x (\lc{x}_{\mbox{\it\tiny T}}\ \!\varphi)$. Then  ${\sf Con}_{\mbox{\it\tiny T}}$ is $\neg\Box_{\mbox{\it\tiny T}}\bot$. We drop the subscript when $T=\pa$.

If \pa\ selector proves the consistency of a theory $T$, then for some selector $s(x)$, 
\[ \pa \proves \forall x [s(x)\!:\!\neg x\!:_T\!\bot] .
\]
After weakening this claim by switching from ``$s(x)\!:$" to ``$\Box$", we get 
\[ \pa \proves \forall x \Box\neg x\!:_T\!\bot .
\]
For $T=\pa\ + {\sf Con}_{\sf PA}$, by internalized {\it Modus Ponens}, we obtain 
\begin{equation}\label{?!?}
\pa\proves\forall x\Box\neg\lc{x}\Box\bot. 
\end{equation}
This yields 
\begin{equation}\label{???}
\pa\proves\Box\Box\bot\imp\Box\bot. 
\end{equation}
Indeed, reason in \pa\ and assume $\Box\Box\bot$, i.e., 
$\exists x(\lc{x}\Box\bot)$. 
By a strong form of provable $\Sigma_1$-completeness,
$$
\pa\proves \lc{x}\Box\bot \imp\Box \lc{x}\Box\bot,
$$
and we would have $\exists x\Box \lc{x}\Box\bot$. From this and (\ref{?!?}), we get $\Box\bot$. However, it is easy to show that (\ref{???}) does not hold. So, \pa\ cannot selector prove the consistency of $\pa + {\sf Con}_{\sf PA}$.

\subsection{Consistency of extensions of \pa}
Consider a theory $T$ in the language of arithmetic such that $T\supset \pa$. Does $T$ prove its own consistency? Yes, it does.
Indeed, consider a derivation $D$ in $T$ and let $\widetilde{D}$ be the conjunction of universal closure of all formulas from $D$. For an appropriate $n$, 
\[ \pa\proves \widetilde{D}\imp \mbox{\it Tr}_n(\widetilde{D}).\] 
Since $T\proves \widetilde{D}$, 
\[ T\proves  \mbox{\it Tr}_n(\widetilde{D}).\] 
Since $T$ proves  $\neg\mbox{\it Tr}_n(\bot)$, $T$ proves that $\bot$ is not in $D$. 
\mpn
This reasoning defines a primitive recursive selector that for each $D$ builds a $T$-proof of {\it``$\bot$ is not in $D$."} The natural internalization of this reasoning in \pa\ (hence in $T$), confirms that this is a selector proof of $T$-consistency in $T$. 
\mpn
Selector proofs methods seem to apply to {\sf ZF} and its extensions. Note that {\sf ZF} is \textit{essentially reflexive}. Namely, any theory $T\supseteq {\sf ZF}$ proves the standard consistency formula for each of its finite fragments $\{\varphi\}$: given $\varphi$ we constructively build the proof of $\mbox{\sf Con}_{\{\varphi\}}$ in $T$. 
This suggests a natural selector proof of consistency of any such $T$ in $T$: given a derivation $\cal D$ in $T$, calculate the conjunction $\widetilde{\cal D}$ of all formulas in $\cal D$ and find a proof of $\mbox{\sf Con}_{\{\widetilde{\cal D}\}}$ in $T$. This yields that $T$ proves that $\cal D$ does not contain contradictions. This proof is naturally formalizable in {\sf ZF}, hence in $T$. 
\mpn
As before, the metamathematical value of such a proof depends on how much trust we invest into $T$ itself. 
Speaking informally, taking $T$ as a mathematical ``trusted core," we can establish its consistency without extending this trusted core. Under the traditional unprovability of consistency paradigm, we cannot prove the consistency of $T$ without invoking new principles, which leads to an unlimited growth of the trusted core, the so-called ``reflection tower."
Our findings eliminate such ``reflection towers."   Moreover, we can envision a new norm in the foundations ``a sufficiently rich theory can formally establish its own consistency."

\section{Basic theory of selector proofs}\label{shemy}

Mathematical practice operates with serial properties such as induction, product of polynomials, propositional tautologies, etc., for which there were given selector proofs within contentual arithmetic, but no special theory of selector proofs has been developed.  The study of consistency as a serial property required a deeper analysis. Let us sketch some general theory of selector proofs in \pa\ for serial properties generated by an arithmetical formula. 

\subsection{Arithmetical schemes and their proofs}\label{5_1}

\begin{definition}
Let $\varphi(x)$ be an arithmetical formula with a designated variable $x$. By $\{\varphi(x)\}$ we denote a serial property
\[ \{ \varphi(0), \varphi(1), \varphi(2), \ldots ,\varphi(n), \ldots  \} \]
which we call {\em scheme}. 
From Section~\ref{sel}, a {\em proof of a scheme $\{\varphi(x)\}$} in \pa\ is a pair $\langle s, v\rangle$ where 
\begin{itemize}
\item $s$ is a primitive recursive term (\textit{selector}),
\item $v$ is a \pa-proof of  $\ \forall x [s(x)\!:\!\varphi(x)]$ (\textit{verifier}).
\end{itemize} 
\end{definition}
We have shown that the {\em consistency scheme} $\{\neg\lc{x}\bot\}$ which is a serial property 
\[ \{\neg\lc{0}\bot, \neg\lc{1}\bot, \neg\lc{2}\bot, \ldots, \neg\lc{n}\bot, \ldots \},\]
 has a proof in \pa. 
\medskip\par
It is easy to observe that the following properties of proofs of schemes hold. 
\begin{itemize}
\item Proofs of schemes in \pa\ are finite syntactic objects. 
\item The proof of a scheme in \pa\ predicate
\[ \mbox{``$\langle s, v\rangle$ is a proof of scheme $\{\varphi\}$"}
\]  
is decidable. 
\item The set of provable in \pa\ schemes is recursively enumerable.
\end{itemize}
\begin{definition}
A scheme $\{\varphi\}$ is 
\begin{itemize}
\item {\em provable in \pa} if it has a proof in \pa, 
\item {\em strongly provable in \pa} if $\pa\proves\forall x \varphi(x)$, 
\item {\em instance provable in \pa} if $\pa\proves \varphi(n)$, for each $n=0,1,2, \ldots$. 
\end{itemize}
\end{definition}

\begin{Prop}\label{p2}
$\ $

\begin{enumerate}
\item ``Strongly provable in \pa" yields ``provable in \pa,"
\item ``provable in \pa" yields ``instance provable in \pa."
\end{enumerate}
\end{Prop}
\begin{proof}
\medskip\par
i) Suppose $q$ is a proof of $\forall x \varphi(x)$, i.e., 
\[  \lc{q}\forall x \varphi(x).\]
By an easy transformation of proofs, find a primitive recursive term $s(x)$ such that \pa\ proves 
\[ \forall x[\lc{s(x)}\varphi(x)].\]
Let $v$ be a \pa-proof of the latter, 
\[ \lc{v}\forall x[\lc{s(x)}\varphi(x)],\]
then $\langle s, v\rangle$ is a proof of the scheme $\{\varphi\}$.
\medskip\par
ii) Let $\langle s, v\rangle$ be a proof of the scheme $\{\varphi\}$ and $n$ be a numeral.  
Since \pa\ proves  $$\forall x[\lc{s(x)}\varphi(x)]$$ we also have that  \pa\ proves $$\lc{s(n)}\varphi(n).$$
Then $s(n)$ is a proof of $\varphi(n)$. Indeed, otherwise $\neg[\lc{s(n)}\varphi(n)]$ holds and, as a p.r. sentence it is provable in \pa. This would yield a specific derivation of a contradiction in \pa\ which was proved to be impossible earlier.  
\end{proof}
\begin{Coro} The proof of {\em Proposition~\ref{p2}(ii)} naturally extends from schemes to all serial properties. Hence \pa\ does not prove serial propertiers containing $\bot$.
\end{Coro}
\begin{Coro} Proving serial properties does not add any new theorems to \pa.
\end{Coro}

\begin{Prop}\label{p3}
$\ $ 
\begin{enumerate}
\item ``Instance provable" does not yield  ``provable,"

\item ``provable" does not yield ``strongly provable."
\end{enumerate}
\end{Prop} 
\begin{proof}
$\ $

i)  Consider the \pa-consistency scheme 
$ \{\neg\lc{x}\bot\}.$ 
As we have shown, this scheme is provable in \pa. 
By {\bf G2}, $\pa\not\proves\forall x \neg\lc{x}\bot$, hence the scheme $\{\neg\lc{x}\bot\}$ is not strongly provable in \pa. 
\medskip\par
ii)  The scheme $$\{\neg\lc{x}\Box\bot\}$$ which we suggest calling the \textit{unprovability of inconsistency scheme}
is instance provable since  $\neg\lc{n}\Box\bot$
is true for each $n=0,1,\ldots$, and hence provable in \pa\ as a true primitive recursive sentence. 
On the other hand, the unprovability of inconsistency scheme is not provable in \pa. Indeed, suppose the scheme $\{\neg\lc{x}\Box\bot\}$ is provable in \pa. Then for some selector term $s$, 
\[ \pa\proves\forall x [\lc{s(x)}\neg\lc{x}\Box\bot]. \]
By an easy \pa-reasoning, we get 
\[ \pa\proves\forall x\Box\neg\lc{x}\Box\bot,\] 
which is impossible by the aforementioned Kurahashi-Sinclaire observation. 
\end{proof}

\subsection{Proof of consistency scheme but not consistency property}\label{ConNotCon}

There are additional subtleties in proving consistency presented as a scheme. Consider the following primitive recursive function $s(x)$ which given $x$ returns a \pa-proof of $\neg\lc{x}\bot$:
\begin{quote}
{\it
Given $x$, check whether $x$ is a proof of $\bot$ in \pa. If ``yes," then put $s(x)$ to be $x$ followed by a simple derivation of $\neg\lc{x}\bot$ from $\bot$. If ``no," then use provable $\Sigma_1$-completeness and put $s(x)$ to be a constructible derivation of $\neg\lc{x}\bot$ in \pa. 
}
\end{quote}
Let $v$ be an obvious \pa-derivation of $\forall x[\lc{s(x)}\neg\lc{x}\bot]$. Consider two questions: 
\begin{enumerate}
\item Whether $\langle s, v\rangle$ is a proof of the scheme $\{\neg\lc{x}\bot\}$ in \pa?
\item Whether $\langle s, v\rangle$ is a formalized selector proof of \pa-consistency? 
\end{enumerate}

The answer to (i) is obviously ``yes" since $\langle s, v\rangle$ fits the definition of a proof of the scheme $\{\neg\lc{x}\bot\}$ in \pa.

The answer to (ii) is ``no." For the affirmative answer, each $s(n)$, as a contentual argument, should be a $\widehat{\pa}$-proof that $n$ does not contain $\bot$. This condition is not met: $s(n)$ only tells us that if $n$ contains $\bot$, we would still be able to offer a fake proof of $\neg\lc{n}\bot$. So, $s(n)$ is not a proof that $n$ does not contain $\bot$.

This example exposes some differences between \pa\ and the contentual arithmetic $\widehat{\pa}$. Whereas formal \pa-derivations encode correct $\widehat{\pa}$-proofs, it can take a contentual judgement to decide whether such proofs satisfy informal requirements. 


\section{Discussion}


The first mathematically complete version of this paper was the preprint \cite{Art19} in which a selector proof of \pa-consistency in \pa\ was given. The current exposition is a slim version of \cite{Art19} which takes into account connections with Hilbert's works on consistency proofs. For further results and surveys cf. \cite{Che21,Che22,Kus21,SSK23,Wil21}.

Representing \pa-consistency by the arithmetical formula ${\sf Con}_{\sf PA}$ has deviated from the original Hilbert's views of consistency which only require that each of the statements 
\begin{equation}\label{standard}
\{\neg\lc{0}\bot, \neg\lc{1}\bot, \neg\lc{2}\bot, \ldots, \neg\lc{n}\bot, \ldots \}
\end{equation}
holds. {\it A priori} the consistency property is concerned about consistency of combinatorial objects -- formal derivations -- represented by their numeral codes. At the same time, the universal quantifier in ${\sf Con}_{\sf PA}$, 
\[ \forall x\neg\lc{x}\bot,\] 
speaking model-theoretically, spans over both, standard and nonstandard numbers and it states something (much) stronger than the desired validity of (\ref{standard}), namely, that  $\neg\lc{x}\bot$ holds for both standard and nonstandard $x$'s. The latter fails in some \pa-models, in which $c\!:\!\bot$ holds for some nonstandard $c$.
 
In contrast, the selector proof respects the original serial character of consistency (\ref{standard}) and the consistency proof in \pa\ boils down to a formal derivation of 
\begin{equation}\label{boils}
\forall x[s(x)\!:\!\neg\lc{x}\bot]
\end{equation}
for an appropriate selector term $s(x)$.
Since the selector is primitive recursive, it is total, and for each standard $n$ returns a standard proof $s(n)$ of $\neg\lc{n}\bot$, which is sufficient for claiming consistency. Accidentally, (\ref{boils}) holds for nonstandard numbers as well, but the ``proofs" $s(c)$ of $\neg c:\!\bot$ for nonstandard $c$'s are themselves nonstandard. 

Our findings also free the verification framework from some ``impossibility" limitations. Imagine that we want to verify a $\Pi_1$-formula
\begin{equation}\label{ver}
\forall x Q(x)
\end{equation}
for some p.r. predicate $Q(x)$ by proving (\ref{ver}) in \pa. 
Given G\"odel's second incompleteness theorem, in addition to provability of (\ref{ver}) in \pa, one needs some consistency assumptions about \pa\ to conclude that $Q(n)$ holds for each $n=0,1,2,\ldots$. Since it was believed that these additional assumptions could not be verified in \pa, this, strictly speaking, left a certain foundational loophole.
In our framework, \pa\ proves its own consistency, and hence these additional meta-assumptions can be dropped. Proving (\ref{ver}) formally in \pa, as we have seen, is certified as a self-sufficient verification method.
\medskip\par
There is a long history of suggestions on bypassing G\"odel's second incompleteness theorem, for example, \cite{Det79,Det86,Det01,Wil21,Zach07}. The approach to proving consistency adopted in this paper is different from them and can be traced back to Hilbert ideas. 

Our findings show the following:
\begin{itemize}
\item Hilbert's epsilon substitution method is a selector proof approach to establishing consistency. Such proofs are not precluded by G\"odel's second incompleteness theorem since they do not aim at deriving the consistency formula. The popular view that G\"odel's theorems invalidate Hilbert's consistency program has been unwarranted.
\item We give a selector proof of \pa-consistency in \pa\ and extend this result to other theories containing \pa. This shows the ``unprovability of consistency paradigm" incorrect. 
\end{itemize}

\section{Thanks}
Thanks to
\medskip\par
Arnon Avron, Lev Beklemishev, Sam Buss, Walter Carnielli, Nachum Dershowitz, Michael Detlefsen, Michael Dunn, Hartry Field, Mel Fitting, Richard Heck, Carl Hewitt, John H. Hubbard, Reinhard Kahle, Karen Kletter, Vladimir Krupski, Taishi Kurahashi, Yuri Matiyasevich, Richard Mendelsohn, Eoin Moore, Andrei Morozov, Larry Moss, Anil Nerode, Elena Nogina, Vincent Peluce, Andrei Rodin, Chris Scambler, Sasha Shen, Richard Shore, Morgan Sinclaire, Thomas Studer, Albert Visser, Dan Willard, Noson Yanofsky, and many others.


\end{document}